\documentclass[10pt]{amsart}

\usepackage[left]{lineno}

\usepackage{amsmath,amssymb,amsthm}
\usepackage[mathscr]{euscript}
\usepackage{color}
\usepackage{hyperref,doi}
\usepackage{epsf}
\usepackage{enumerate}
\usepackage{graphicx}
\usepackage{array}

\usepackage[margin=1in]{geometry}

\newtheorem{theorem}{Theorem}[section]
\newtheorem{prop}[theorem]{Proposition}

\newtheorem{conjecture}[theorem]{Conjecture}
\newtheorem{disconjecture}[theorem]{(Disproved) Conjecture}
\newtheorem{lemma}[theorem]{Lemma}
\newtheorem{corollary}[theorem]{Corollary}
\newtheorem{question}[theorem]{Question}
\theoremstyle{definition}
\newtheorem{definition}[theorem]{Definition}
\newtheorem{remark}[theorem]{Remark}
\newtheorem{example}[theorem]{Example}

\newcommand{\conv}[1]{\mathrm{conv}\left\{#1\right\}}

\newcommand{\C}{\mathbb{C}}
\newcommand{\R}{\mathbb{R}}
\newcommand{\Z}{\mathbb{Z}}

\newcommand{\Frac}[2]{\genfrac{}{}{}{}{#1}{#2}}

\newcommand{\Pq}{\Delta_{(1,q)}}
\newcommand{\Pp}{\Delta_{(1,p)}}
\newcommand{\cone}[1]{\mathrm{cone}\left(#1\right)}
\DeclareMathOperator{\lcm}{lcm}
\DeclareMathOperator{\ehr}{Ehr}
\DeclareMathOperator{\Vol}{Vol}

\newcommand\commentout[1]{}

\begin{document}



\title[Detecting the Integer Decomposition Property]{Detecting the Integer Decomposition Property and Ehrhart Unimodality in Reflexive Simplices}

\author{Benjamin Braun}
\address{Department of Mathematics\\
         University of Kentucky\\
         Lexington, KY 40506--0027}
\email{benjamin.braun@uky.edu}

\author{Robert Davis}
\address{Department of Mathematics\\
         Michigan State University\\
         East Lansing, MI 48824}
\email{davisr@math.msu.edu}

\author{Liam Solus}
\address{Matematik\\
         KTH\\
         Stockholm, Sweden}
\email{solus@kth.se}



\date{31 May 2018}

\thanks{
Benjamin Braun was partially supported by grant H98230-16-1-0045 from the U.S. National Security Agency.
Liam Solus was partially supported by an NSF Division of Mathematical Sciences Postdoctoral Fellowship (DMS - 1606407).
We thank the anonymous referees for their thoughtful and helpful feedback.
}

\begin{abstract}
A long-standing open conjecture in combinatorics asserts that a Gorenstein lattice polytope with the integer decomposition property (IDP) has a unimodal (Ehrhart) $h^\ast$-polynomial.  
This conjecture can be viewed as a strengthening of a previously disproved conjecture which stated that any Gorenstein lattice polytope has a unimodal $h^\ast$-polynomial.  
The first counterexamples to unimodality for Gorenstein lattice polytopes were given in even dimensions greater than five by Musta{\c{t}}{\v{a}} and Payne, and this was extended to all dimensions greater than five by Payne.  
While there exist numerous examples in support of the conjecture that IDP reflexives are $h^\ast$-unimodal, its validity has not yet been considered for families of reflexive lattice simplices that closely generalize Payne's counterexamples.
The main purpose of this work is to prove that the former conjecture does indeed hold for a natural generalization of Payne's examples.
The second purpose of this work is to extend this investigation to a broader class of lattice simplices, for which we present new results and open problems.
\end{abstract}

\maketitle



\section{Introduction}


Two well-studied properties of the coefficients of polynomials with non-negative integer coefficients are symmetry and unimodality.
 A polynomial $a_0+a_1z+\cdots+a_dz^d$ is called \emph{symmetric} (\emph{with respect to $d$}) if $a_i = a_{d-i}$ for all $i\in[d]$. 
 A polynomial $a_0+a_1z+\cdots+a_dz^d$ is called \emph{unimodal} if there exists an index $j$ such that $a_i\leq a_{i+1}$ for all $i<j$ and $a_{i}\geq a_{i+1}$ for all $i\geq j$.  

Given a finitely generated $\C$-algebra $A=\oplus_{i=0}^\infty A_i$, we say $A$ is \emph{standard} if $A$ is generated by $A_1$ and \emph{semistandard} if $A$ is integral over the subalgebra generated by $A_1$.
The Hilbert series of a graded algebra $A$ is the generating function
\[
H(A;z):=\sum_{i=0}^\infty \dim_{\C}(A_i) z^i =\frac{\sum_{j} h_jz^j}{(1-z)^{\dim(A)}} \, .
\]
The polynomial $h(A;z):=\sum_{j} h_jz^j$ is called the \emph{$h$-polynomial} of $A$.
When $A$ is Cohen-Macaulay, each $h_j\in \Z_{\geq 0}$.
An active topic of research is developing combinatorial interpretations of the $h$-polynomial coefficients and their distribution.  
It is known that the $h$-polynomial of $A$ is symmetric when $A$ is Gorenstein, but the unimodality property is more subtle.
A long-standing open problem is the following conjecture posed by Brenti \cite[Conjecture 5.1]{brentisurvey}, inspired by a conjecture of Stanley~\cite{stanleylogconcave}.
It also appears in \cite[Conjecture 1.5]{hibiflawless}.
\begin{conjecture}[Brenti, \cite{brentisurvey}]
\label{conj:stanleygor}
For a standard graded Gorenstein integral domain $A$, $h(A;z)$ is unimodal.
\end{conjecture}

A subset $P\subset\R^n$ is a $d$-dimensional \emph{convex lattice polytope} if $P$ arises as the convex hull of finitely many points in $\Z^n$ that together span an affine $d$-flat in $\R^n$.  
The \emph{Ehrhart function} of $P$ is the lattice point enumerator $i(P;t):=|tP\cap\Z^n|$, where $tP:=\{tp:p\in P\}$ denotes the $t^{th}$ dilate of the polytope $P$.  
It is well-known \cite{Ehrhart} that $i(P;t)$ is a polynomial in $t$ of degree $d$, and the corresponding \emph{Ehrhart series} of $P$ is the rational function
\[
\ehr_P(z) := \sum_{t\geq0}i(P;t)z^t = \frac{h_0^*+h_1^*z+\cdots+h_d^*z^d}{(1-z)^{d+1}},
\]
where the coefficients $h^\ast_0,h^\ast_1,\ldots,h^\ast_d$ are all nonnegative integers \cite{StanleyDecompositions}.  
The polynomial $h^\ast(P;z):=h_0^*+h_1^*z+\cdots+h_d^*z^d$ is called the \emph{(Ehrhart) $h^\ast$-polynomial} of $P$.  
Analogously, the coefficient vector $h^\ast(P) := (h_0^\ast,h_1^\ast,\ldots,h_d^\ast)$ is called the \emph{(Ehrhart) $h^\ast$-vector} of $P$. 
We will often say $P$ is \emph{$h^\ast$-unimodal} whenever $h^\ast(P;z)$ is unimodal.  
Since $h^\ast(P;z)$ arises via the enumeration of combinatorial data and has only nonnegative integer coefficients, researchers have studied combinatorial interpretations of the $h^\ast$-coefficients and their distribution.  
A current challenge in Ehrhart theory is to understand the geometric properties of a polytope that are necessary and/or sufficient for $h^\ast(P;z)$ to be simultaneously symmetric and unimodal \cite{BraunUnimodalSurvey}.

While Ehrhart series arise from a combinatorial context in polyhedral geometry, there is a well-known connection to commutative algebra.
The \emph{cone over $P$} is 
\[
\cone{P}:=\text{span}_{\R_{\geq 0}}\{(1,p):p\in P\}\subset \R\times \R^n,
\]
where we consider the new variable to be indexed at $0$, i.e. $w\in \R\times\R^n=\R^{1+n}$ is written $w=(w_0,w_1,\ldots,w_n)$.
Given a lattice polytope $P$ in $\R^n$, the \emph{semigroup algebra associated to $P$} is $\C[P]:=\C[x^w:w\in \cone{P}\cap \Z^{1+n}]$.
We grade $C[P]$ by $\deg(x_0^{w_0}\cdots x_n^{w_n})=w_0$.
With this grading, $\C[P]$ is a semistandard semigroup algebra, where the algebra generated by $\C[P]_1$ is $\C[x^{w}:w\in (1,P)\cap \Z^{1+n}]$; i.e., the semigroup algebra generated by integer points in $(1,P)$.
Given this, Conjecture~\ref{conj:stanleygor} manifests itself in the setting of polytopes that have symmetric $h^\ast$-polynomials and the integer decomposition property, which is defined as follows.

\begin{definition}
A polytope $P$ has the {\em integer decomposition property}, or {is IDP}, if for every positive integer $m$ and each $w \in mP \cap \Z^n$, there exist $m$ points $x_1,\ldots,x_m \in P \cap \Z^n$ for which $w = \sum x_i$.
\end{definition}
It is straightforward to verify that a lattice polytope $P$ is IDP if and only if $\C[P]$ is standard.
Regarding symmetry, a lattice polytope $P \subseteq \R^n$ containing $0$ in its interior is called \emph{reflexive} if its polar $P^* := \{ y \in \R^n\ : \ x^Ty \leq 1 \text{ for all } x \in P\}$ is also a lattice polytope.
In \cite{HibiDualPolytopes}, it is shown that a lattice polytope $P\subset\R^n$ is reflexive if and only if $P$ is full-dimensional (i.e. $d=n$), the origin is in its interior, and $h^\ast(P;z)$ is symmetric.  
Reflexive polytopes are a special case of the full-dimensional polytopes $P\subset\R^n$ for which $h^\ast(P;z)$ is symmetric, known as Gorenstein polytopes.
In this setting, Conjecture~\ref{conj:stanleygor} translates to the following statement, which is often attributed to Hibi and Ohsugi~\cite{hibiohsugiconj}.
\begin{conjecture}
\label{conj:hibiohsugi}
If $P$ is Gorenstein and IDP, then $h^\ast(P;z)$ is unimodal.
\end{conjecture} 
As in the case of graded Cohen-Macaulay algebras, unimodality of $h^\ast(P;z)$ has proven to be a more elusive property to characterize than symmetry when using only the geometry of $P$.
Since the symmetry property of $h^\ast(P;z)$ reduces the number of inequalities to be verified by one-half, there has been much research into when $h^\ast(P;z)$ is both symmetric and unimodal~\cite{BraunUnimodalSurvey}. 
It was shown by Musta{\c{t}}{\v{a}} and Payne~\cite{MustataPayne} that there exist reflexive polytopes whose $h^\ast$-polynomials are not unimodal.  
These results were then extended to reflexive simplices in every dimension greater than five by Payne \cite{Payne}.
This disproved the following (earlier) conjecture of Hibi~\cite{Hibi}.
\begin{disconjecture}
[Hibi, \cite{Hibi}]
\label{conj: hibi}
If $P$ is Gorenstein then $h^\ast(P;z)$ is unimodal.  
\end{disconjecture}
Thus, Conjecture~\ref{conj:hibiohsugi} is a reasonable strengthening of Conjecture~\ref{conj: hibi} to consider in light of the counterexamples presented by Payne \cite{Payne} and its previous proposal in the algebraic context by Brenti (Conjecture~\ref{conj:stanleygor}).
Indeed, one may show that none of the counterexamples provided by Musta{\c{t}}{\v{a}} and Payne are IDP. 
While the literature is ripe with examples of lattice polytopes supporting Conjecture~\ref{conj:hibiohsugi}, surprisingly, there has been little investigation as to whether or not the conjecture holds for reflexive lattice simplices that closely generalize the counterexamples constructed by Payne.  
The purpose of the present paper is to carefully identify a natural generalization of Payne's examples, and prove that Conjecture~\ref{conj:hibiohsugi} does indeed hold for this family of reflexive lattice simplices. 

The generalization of Payne's examples that we will consider here is a subcollection $\mathcal{Q}$ of the reflexive lattice simplices whose associated toric varieties are weighted projective spaces.
Consequently, each simplex $\Delta\in\mathcal{Q}$ can be associated to an integer partition $q(\Delta)$ by using a classification system for reflexive simplices developed by Conrads \cite{conrads}.  
The collection $\mathcal{Q}$ can be stratified by the number of parts used in the integer partitions $q(\Delta)$.  
We will see that $q(\Delta)$ for Payne's examples have only two distinct parts, making the family of all $\Delta$ where $q(\Delta)$ has only two distinct parts a close generalization of Payne's counterexamples.
Our main result is to show that Conjecture~\ref{conj:hibiohsugi} holds for this family.

The remainder of this paper is outlined as follows: 
In Subsection~\ref{subsec: the reflexive simplices}, we recall the counterexamples to Conjecture~\ref{conj: hibi} presented by Payne in \cite{Payne}, and we present the generalizing family $\mathcal{Q}$.  
In Subsections~\ref{subsec: h*-polynomials} and~\ref{subsec: IDP for q's}, we present a formula for the $h^\ast$-polynomial and a characterization of IDP for all simplices in $\mathcal{Q}$, respectively.  
In Section~\ref{sec: q-vectors with fixed support}, we prove some general results on IDP and $h^\ast$-unimodality for simplices in $\mathcal{Q}$ whose associated integer partitions $q(\Delta)$ have a fixed set of parts. 
In Section~\ref{sec: support on two integers}, we prove Conjecture~\ref{conj:hibiohsugi} for our closest generalization of Payne's examples; i.e., for all simplices $\Delta\in\mathcal{Q}$ whose associated partitions $q(\Delta)$ have two parts.
We then end with a discussion of future directions in Section~\ref{subsec: future directions}.


\section{IDP and Ehrhart Theory for the Reflexive Simplices $\Pq$} 
\label{sec: idp for reflexive simplices}

\subsection{The reflexive simplices $\Delta_{(1,q)}$}
\label{subsec: the reflexive simplices}
Given a polytope $P \subseteq \R^n$ with vertices $v_0,\ldots,v_k$, the {\em height} of a point $w \in \cone{P}$ is its zero-th coordinate $w_0$.
If $P$ is a lattice simplex with vertices $v_0,v_1,\ldots,v_n$, the \emph{fundamental parallelepiped} of $\cone{P}$ is
\[
\Pi_P:=\left\{\sum_{i=0}^n\lambda_i(1,v_i)\ :\ 0\leq \lambda_i<1  \right\}\, .
\]
The fundamental parallelepiped tiles $\cone{P}$ through nonnegative integer combinations of $\{(1,v_i)\}_{i=0}^n$.  
Consequently, if $P$ is a lattice simplex then 
\begin{equation}
\label{eqn: fpp}
h^\ast(P;z)=\sum_{w\in \Pi_P\cap \Z^n}z^{w_0} \, .
\end{equation}

In this paper, we will focus on a family of reflexive simplices $\mathcal{Q}$ that can be defined as follows.  
Let $q=(q_1,q_2,\ldots,q_n)$ be a weakly increasing sequence of positive integers satisfying the condition
\[
	q_j \mid (1+\sum_{i \neq j}q_i ) 
\]
for all $j=1,\ldots,n$.
For such a vector, the simplex 
\[
	\Pq:=\conv{e_1,e_2,\ldots,e_n,-\sum_{i=1}^nq_ie_i}, \, 
\]
where $e_i\in \R^n$ is the $i$-th standard basis vector, is reflexive.
When the $q$-vector is understood, we will often label the vertices of $\Pq$ as $v_i := e_i$ and $v_0 := -\sum_{i=1}^nq_ie_i$.  
The collection of simplices $\mathcal{Q}$ is important in algebra and geometry since it is contained in the family of simplices whose associated toric varieties are weighted projective spaces \cite{conrads}.  
Combinatorially, they are significant since they contain the counterexamples to Conjecture~\ref{conj: hibi} developed by Payne \cite{Payne}.  
In particular, Payne showed that for integers $r\geq0$, $s\geq 3$, and $k\geq r+2$ the reflexive simplex $\Delta_{(1,q)}$ for 
\begin{equation}
\label{eqn: payne}
q=(\underbrace{1,1,\ldots,1}_{sk-1 \text{ times}},\underbrace{s,s,\ldots,s}_{r+1 \text{ times}}) \, 
\end{equation}
is not $h^\ast$-unimodal.

It follows from (\ref{eqn: fpp}) that the \emph{normalized volume} of $\Delta_{(1,q)}$ (i.e. the value $h^\ast(\Delta_{(1,q)};1) = n!\Vol(\Delta_{(1,q)})$) is equal to $1+q_1+\cdots+q_n$ \cite[Proposition 4.4]{NillSimplices}.  
Consequently, the simplices $\Delta_{(1,q)}$ are naturally stratified via their normalized volume by way of \emph{integer partitions}; i.e., for every $\Delta_{(1,q)}$ there exists a partition of the integer $n!\Vol(\Delta_{(1,q)})-1$ with parts $r_1<r_2<\cdots< r_k$ such that
\[
q = (r_1^{x_1},r_2^{x_2},\ldots,r_k^{x_k}):=(\underbrace{r_1,r_1,\ldots,r_1}_{x_1\text{ times}},\underbrace{r_2,r_2,\ldots,r_2}_{x_2\text{ times}},\ldots,\underbrace{r_k,r_k,\ldots,r_k}_{x_k\text{ times}}).
\]
So as to speak formally from this perspective, we make the following definition.
\begin{definition}
\label{def:support}
We say that both $q$ and $\Pq$ are \emph{supported} by the vector $r = (r_1,\ldots,r_k)$ (or the integers therein) if there exist positive integers $r_1<r_2<\cdots <r_k$ and $x_1,\ldots,x_k$ such that 
\[
q=(q_1,\ldots,q_n)=(r_1^{x_1},r_2^{x_2},\ldots,r_k^{x_k}).
\]
An \emph{$r$-vector} is any vector of positive integers $r = (r_1,\ldots,r_k)$ in which $r_1<r_2<\cdots <r_k$.  
\end{definition}
It follows that the natural generalization of Payne's examples are those $\Delta_{(1,q)}$ supported by two integers.  
In Section~\ref{sec: support on two integers}, we will prove Conjecture~\ref{conj:hibiohsugi} for this generalization of Payne's examples.  
In Section~\ref{sec: q-vectors with fixed support}, we will prove some results about $q$-vectors with a fixed support $r$.


\subsection{The $h^\ast$-polynomials for $\Pq$}
\label{subsec: h*-polynomials}

In this subsection, we establish basic properties of $\Pq$ with respect to $h^\ast$-polynomials.
We begin with a formula for $h^\ast(\Delta_{(1,q)};z)$ in terms of the entries of the vector $q$.  

\begin{theorem}
\label{thm:hstarq}
The $h^\ast$-polynomial of $\Pq$ is 
\[
h^\ast(\Pq;z) = \sum_{b=0}^{q_1+\cdots +q_n}z^{w(b)}
\]
where 
\[
w(b):=b-\sum_{i=1}^n\left\lfloor\frac{q_ib}{1+q_1+\cdots +q_n} \right\rfloor \, .
\]
\end{theorem}

\begin{proof}
Let $(1,v_1),(1,v_2),\ldots,(1,v_n)$, and $(1,v_0)$ denote, respectively, the columns of 
\[
\left[
\begin{array}{cccccc}
1 & 1 & 1 &  \cdots & 1 & 1 \\
1 & 0 & 0 & \cdots & 0 & -q_1 \\
0 & 1 & 0 & \cdots & 0 & -q_2 \\
0 & 0 & 1 & \cdots & 0 & -q_2 \\
\vdots & \vdots & \vdots & \ddots & \vdots & \vdots \\
0 & 0 & 0 & \cdots & 1 & -q_n 
\end{array}\right]\, .
\]
By \cite[Proposition 4.4]{NillSimplices} and Cramer's rule, every point in $\Pi_{\Pq}$ is of the form
\[
p:=\sum_{i=0}^n\lambda_i(1,v_i)
\]
where $0\leq \lambda_i<1$ for all $i$ and $\lambda_i=\frac{b_i}{1+q_1+\cdots + q_n}$ for some $b_i=0,1,\ldots,q_1+\cdots +q_n$.
If $p \in \Z^{1+n}$, then it must follow that for $i \geq 1$
\[
\lambda_i=q_i\lambda_0-\lfloor q_i\lambda_0 \rfloor \, .
\]
Hence, the choice of $\lambda_0$ determines the point $p$, and we write 
\[
\lambda_0:=\frac{b}{1+q_1+\cdots + q_n} \, .
\]
The height of the point corresponding to a given value of $b$ is the first coordinate, which is easily computed to be
\[
b-\sum_{i=1}^n\left\lfloor\frac{q_ib}{1+q_1+\cdots +q_n} \right\rfloor \, .
\]
Recalling equation~(\ref{eqn: fpp}) completes the proof.
\end{proof}


\subsection{The Integer Decomposition Property and $\Pq$}
\label{subsec: IDP for q's}

The following theorem (Theorem~\ref{thm:icreflexives}) provides a characterization of IDP reflexive $\Delta_{(1,q)}$ in terms of the vector $q$.
Our main tool for investigating $\Pq$ is given by Corollary~\ref{cor: reduction of theorem 4.1} below, a necessary but not sufficient relaxation of Theorem~\ref{thm:icreflexives}.

\begin{theorem}
\label{thm:icreflexives}
The reflexive simplex $\Pq$ is IDP if and only if for every $j=1,\ldots,n$, for all $b=1,\ldots,q_j-1$ satisfying
\begin{equation}
\label{eqn 1}
b\left(\Frac{1+\sum_{i\neq j}q_i}{q_j} \right) - \sum_{i\neq j}\left\lfloor \Frac{bq_i}{q_j} \right\rfloor \geq 2 
\end{equation}
there exists a positive integer $c<b$ satisfying the following equations, where the first is considered for all  $1\leq i\leq n$ with $i\neq j$:
\begin{equation}
\label{eqn 2}
\left\lfloor \Frac{bq_i}{q_j} \right\rfloor - \left\lfloor \Frac{cq_i}{q_j} \right\rfloor = \left\lfloor \Frac{(b-c)q_i}{q_j} \right\rfloor, \text{  and}
 \end{equation}
 \begin{equation}
 \label{eqn 3}
c\left(\Frac{1+\sum_{i\neq j}q_i}{q_j} \right) - \sum_{i\neq j}\left\lfloor \Frac{cq_i}{q_j} \right\rfloor = 1.
\end{equation}
\end{theorem}

\begin{proof}
Recall that IDP for a lattice simplex $P$ is equivalent to the property that every lattice point in $\Pi_{P}$ arises as a sum of lattice points in $(1,P)$.
Let $g\in \Pi_{\Pq}$.
Since $\Pq$ is reflexive, we may subtract $(1,0,0,\ldots,0)$ from $g$ until we reach a lattice point $p$ on the boundary of $\cone{\Pq}$.
This point $p$ must lie in the fundamental parallelepiped for a face of $\Pq$. 
Hence, $\Pq$ is IDP if and only if every facet of $\Pq$ is IDP.

Observe that the facet with vertices $\{e_1,\ldots,e_n\}$ is unimodular, and thus it is IDP.
The remaining facets are of the form 
\[
F_j:=\conv{e_1,\ldots,e_{j-1},e_{j+1},\ldots,e_n,-\sum_{i=1}^nq_ie_i}
\]
for $j=1,\ldots,n$.
Thus we are looking for necessary and sufficient conditions for $F_j$ to be IDP.
For $i\in [n]\setminus \{j\}$, set $v_i:=e_i$, and set $v_0:=-\sum_{i=1}^nq_ie_i$.
Every lattice point in $\Pi_{F_j}$ is of the form
\[
r=\sum_{\substack{0\leq i \leq n \\ i\neq j}}\lambda_i(1,v_i) \, .
\]
Since the vector $(1,e_j)$ is not a summand for $r$, and since $r$ is a lattice point, $\lambda_0\cdot (-q_j)$ must be an integer.
Hence,
$
\lambda_0=b/q_j
$
for some $b$ in $\{0,1,2,\ldots,q_j-1\}$.
Note that if $b=0$, $r$ is in the fundamental parallelepiped for the facet with vertices $\{e_1,\ldots,\widehat{e_j},\ldots,e_n\}$, and hence equal to the zero vector.

We next show that for every choice of $b$ between $1$ and $q_j-1$, setting $\lambda_0=b/q_j$ yields a unique lattice point of $\Pi_{F_j}$.
Since every entry of $r$ is integral and $0\leq \lambda_i <1$, it follows that 
\[
\lambda_i=\frac{bq_i}{q_j}-\left\lfloor \frac{bq_i}{q_j} \right\rfloor \, .
\]
Thus, as in the proof of Theorem~\ref{thm:hstarq}, the value of $\lambda_0$ determines the remaining $\lambda_i$'s.
When $\lambda_0=b/q_j$, the height of the resulting point is
\[
\frac{b}{q_j}+\sum_{\substack{1\leq i \leq n \\ i\neq j}}\left( \frac{bq_i}{q_j}-\left\lfloor \frac{bq_i}{q_j} \right\rfloor \right) \, ,
\]
which is an integer if and only if
\[
\frac{b}{q_j}+\sum_{\substack{1\leq i \leq n \\ i\neq j}} \frac{bq_i}{q_j} = b\left(\frac{1+\sum_{\substack{1\leq i \leq n \\ i\neq j}}q_i}{q_j}\right) 
\]
is an integer.
Since $\Pq$ is reflexive, we know that $q_j$ divides $\displaystyle 1+\sum_{\substack{0\leq i \leq n \\ i\neq j}}q_i$, and thus this height is an integer.
Hence, for every choice of $b$ above, $r$ is a lattice point given by
\[
r = r_b := \left[
\begin{array}{c}
b \left(\frac{1+\sum_{i\neq j}q_i}{q_j}\right)-\sum_{ i\neq j}\left(\left\lfloor \frac{bq_i}{q_j} \right\rfloor \right) \\ [2mm]
-\left\lfloor \frac{bq_1}{q_j} \right\rfloor \\
\vdots \\
-\left\lfloor \frac{bq_{j-1}}{q_j} \right\rfloor \\ [3mm]
-b \\ [1mm]
-\left\lfloor \frac{bq_{j+1}}{q_j} \right\rfloor \\
\vdots \\
-\left\lfloor \frac{bq_n}{q_j} \right\rfloor 
\end{array}
\right] \, .
\]
The simplex $F_j$ is IDP if and only if for every $r_b$ at height greater than or equal to two, there exists an $r_c\in \Pi_{F_j}$ at height one and an $r_d\in \Pi_{F_j}$ such that
$
r_b-r_c=r_d \, .
$
Given $b$ and considering this vector equation entry-by-entry, this is equivalent to solving the following system with integers $c$ and $d$ between $1$ and $q_j-1$:
\begingroup
\addtolength{\jot}{1em}
\begin{align}
\label{eqn:sub} -1 + \left( b \left(\frac{1+\sum_{i\neq j}q_i}{q_j}\right)-\sum_{ i\neq j}\left\lfloor \frac{bq_i}{q_j} \right\rfloor \right) & = d \left(\frac{1+\sum_{i\neq j}q_i}{q_j}\right)-\sum_{ i\neq j}\left\lfloor \frac{dq_i}{q_j} \right\rfloor,  \\ 
\label{eqn:floor}  \left\lfloor \frac{bq_{i}}{q_j} \right\rfloor  - \left\lfloor \frac{cq_{i}}{q_j} \right\rfloor & = \left\lfloor \frac{dq_{i}}{q_j} \right\rfloor \, , \, \,  i\neq j, \mbox{ and}\\ 
\label{eqn:unit} c \left(\frac{1+\sum_{i\neq j}q_i}{q_j}\right)-\sum_{ i\neq j}\left\lfloor \frac{cq_i}{q_j} \right\rfloor &= 1.
\end{align}
\endgroup

Substituting \eqref{eqn:unit} for the ``$1$'' on the left-hand side of \eqref{eqn:sub}, and substituting \eqref{eqn:floor} for each of the ``$\left\lfloor \frac{dq_i}{q_j} \right\rfloor$'' on the right-hand side of \eqref{eqn:sub}, yields $b-c=d$.
Thus, we can reduce this system of equations to an equivalent system of equations:
\begingroup
\addtolength{\jot}{1em}
\begin{align}
\label{eqn:floor2}  \left\lfloor \frac{bq_{i}}{q_j} \right\rfloor  - \left\lfloor \frac{cq_{i}}{q_j} \right\rfloor & = \left\lfloor \frac{(b-c)q_{i}}{q_j} \right\rfloor \, , \, \, i\neq j, \mbox{ and}\\ 
\label{eqn:unit2}   c \left(\frac{1+\sum_{i\neq j}q_i}{q_j}\right)-\sum_{ i\neq j}\left\lfloor \frac{cq_i}{q_j} \right\rfloor  &= 1.
\end{align}
\endgroup
Hence, $\Pq$ is IDP if and only if each of its facets is IDP, and this occurs if and only if for every $j=1,\ldots,n$, for all $b=1,\ldots,q_j-1$ corresponding to a lattice point at height at least two there exists a positive integer $c$ solving \eqref{eqn:floor2} and \eqref{eqn:unit2}.
\end{proof}

At this point, we will frequently write $\left\{\frac{a}{b}\right\}$ to denote the fractional part of $a/b$.

\begin{corollary}
\label{cor: reduction of theorem 4.1}
Let $\Pq$ be reflexive.  
If $\Pq$ is IDP then for all $j=1,2,\ldots,n$
\[
\frac{1}{q_j}+\sum_{i\neq j}\left\{\frac{q_i}{q_j}\right\}=1.
\]
For any vector $q$ of increasing positive integers that satisfies these equations for all $j=1,2,\ldots,n$, the simplex $\Pq$ is reflexive.
However, this condition is necessary but not sufficient for $\Pq$ to be IDP.
\end{corollary}

\begin{proof}
By Theorem~\ref{thm:icreflexives}, if $b=1$ satisfies~\eqref{eqn 1} then trivially no $c$ exists satisfying $0<c<b$ and equations~\eqref{eqn 2} and~\eqref{eqn 3}, meaning $\Delta_{(1,q)}$ cannot be IDP.  
Thus, whenever $\Delta_{(1,q)}$ is IDP, $b=1$ must satisfy~\eqref{eqn 3}, which is equivalent to 
\[
\frac{1}{q_j}+\sum_{i\neq j}\left\{\frac{q_i}{q_j}\right\} = 1 \, .
\]
Since this equation implies that the divisibility condition $q_j\mid (1+\sum_{i\neq j}q_i)$ holds, reflexivity of $\Pq$ follows.

To show that this condition is not sufficient to establish IDP, consider the vector $q=(2,2,15,20,20)$.
It is straightforward to compute that the lattice points in $\Pq$ are the columns of the following matrix.
\[
\left[
\begin{array}{ccccccccccccccc}
1 & 0 & 0 & 0 & 0 & -2  & 0 & 0  & 0  & 0  & 0  & 0  & -1  &-1  &-1  \\
0 & 1 & 0 & 0 & 0 & -2  & 0 & 0  & 0  & 0  & 0  & 0  & -1 & -1 &-1 \\
0 & 0 & 1 & 0 & 0 & -15 & 0 & 0  & -1 & -1 & -2 & -3 & -7&-8 &-9 \\
0 & 0 & 0 & 1 & 0 & -20 & 0 & -1 & -1 & -2 & -3 & -4 & -10& -11 &-12 \\
0 & 0 & 0 & 0 & 1 & -20 & 0 & -1 & -1 & -2 & -3 & -4 & -10 &-11 &-12 
\end{array}
\right] \, .
\]
Further, the point 
\[
\left[
\begin{array}{cccccc}
1 & 0 & 0 & 0 & 0 & -2   \\
0 & 1 & 0 & 0 & 0 & -2  \\
0 & 0 & 1 & 0 & 0 & -15 \\
0 & 0 & 0 & 1 & 0 & -20  \\
0 & 0 & 0 & 0 & 1 & -20  
\end{array}
\right] 
\left[
\begin{array}{c}
1/15 \\
1/15 \\
0 \\
2/3 \\
2/3 \\
8/15
\end{array}
\right]
=
\left[
\begin{array}{c}
-1 \\
-1 \\
-8 \\
-10 \\
-10
\end{array}
\right]
\in 
2\Delta_{(1,2,2,15,20,20)} \, .
\]
It is straightforward to verify that this point is not the sum of exactly two lattice points in $\Delta_{(1,2,2,15,20,20)}$, and thus this simplex is not IDP.
However, $(2,2,15,20,20)$ satisfies our linear system.
\end{proof}

\begin{remark}
An equivalent formulation of Corollary~\ref{cor: reduction of theorem 4.1} is that if $q$ corresponds to a reflexive IDP simplex, then for each $j=1,\ldots,n$ the sum
\[
1+\sum_{i\neq j}(q_i\bmod q_j)
\]
must be equal to $q_j$.
In order for the reflexive condition to be satisfied, this sum must be equal to a multiple of $q_j$.
Thus, this necessary condition for IDP reflexive is a strengthening of the divisibility condition characterizing reflexivity.
\end{remark}

\begin{remark}
Corollary~\ref{cor: reduction of theorem 4.1} implies that for IDP $\Delta_{(1,q)}$ and $b=2,3,\ldots,q_j-1$ the choice of $c=1$ is always a potential solution to~\eqref{eqn 2}, as it satisfies equation \eqref{eqn 3}.  
\end{remark}

\begin{example}[Revisiting Payne's simplices]
Recall that Payne presented the simplices $\Delta_{(1,q)}$ with $q$ as defined in equation~(\ref{eqn: payne}) as counterexamples to Conjecture~\ref{conj: hibi}.  
Using Proposition~\ref{thm:hstarq}, we see that
\[
h^\ast(\Pq;z)=\sum_{b=0}^{s(k+r+1)-1}z^{b-(r+1)\left\lfloor sb/s(k+r+1) \right\rfloor} 
\]
from which it is straightforward to compute that
\[
h^\ast(\Pq;z)=(1+z^k+z^{2k}+\cdots +z^{(s-1)k})(1+z+z^2+\cdots +z^{k+r})\, .
\]
Non-unimodality for most of these simplices follows immediately from the conditions on $r$, $b$, and $s$.
One can verify that $\Pq$ is not IDP using the property that $h^\ast_1=1$.  
Since this implies that the only lattice points in $\Delta_{(1,q)}$ are the vertices and the unique interior point of $\Pq$.  
Hence, the vector $(1,0,0,\ldots,0)$ is the only vector in $\Pi_{\Pq}$ at height $1$.
This prevents the existence of the ``$c$'' value required in Theorem~\ref{thm:icreflexives}.
\end{example}


\section{$q$-Vectors With Fixed Support}
\label{sec: q-vectors with fixed support}
In Subsection~\ref{subsec: the reflexive simplices}, we established a stratification of the simplices $\Delta_{(1,q)}$ in terms of their support vectors $r$.  
In this section, we present some first results on the IDP and $h^\ast$-unimodality conditions for $\Delta_{(1,q)}$ from this perspective.  
Our observations in this section center around the interplay between $\Delta_{(1,q)}$ admitting an \emph{affine free sum decomposition} into lower-dimensional $\Delta_{(1,q)}$ and exhibiting IDP and/or $h^\ast$-unimodality.  
As described below, admitting the former property can allow one to recursively detect the latter properties.  
In this section, we will show that only finitely many reflexive IDP $\Delta_{(1,q)}$ with support vector $r$ fail to admit such a free sum decomposition.  
This suggests that the stratification of $\Delta_{(1,q)}$ by their support vectors is a desirable perspective from which to analyze Conjecture~\ref{conj:hibiohsugi} for reflexive $\Delta_{(1,q)}$.  
In particular, we will use these results to prove Conjecture~\ref{conj:hibiohsugi} for $\Delta_{(1,q)}$ supported on two integers in Section~\ref{sec: support on two integers}.  

Our first two results in this section establish that there are infinitely many reflexive simplices supported by each $r$-vector, and that the $r$-vector and multiplicity vector $(x_1,\ldots,x_k)$ can be used to test our necessity condition for IDP. 

\begin{prop}
\label{prop:infinitereflexive}
Fix an $r$-vector $r = (r_1,\ldots,r_k)$  with $\gcd(r_1,\ldots,r_k)=1$.
There are infinitely many reflexive $\Pq$ supported by $r$.
\end{prop}

\begin{proof}
By the positive solution to the Frobenius coin exchange problem \cite[Chapter 1]{BeckRobinsCCD}, there exists a positive integer $M$ such that for all $m\geq M$, there exist positive integers $x_1,\ldots,x_k$ such that $m=1+\sum_{i=1}^kx_ir_i$.
Thus, there exists a positive integer $L$ such that for all $\ell\geq L$, there exist positive integers $x_1,\ldots,x_k$ such that $\ell \lcm(r_1,\ldots,r_k)=1+\sum_{i=1}^kx_ir_i$.
For any such $\ell$, for all $i=1,\ldots,k$ we have that
\[
r_i \, \mid \, \ell \lcm(r_1,\ldots,r_k)=1+\sum_{j=1}^kx_jr_j \, .
\]
Hence, there are infinitely many $q=(r_1^{x_1},r_2^{x_2},\ldots,r_k^{x_k})$ such that the divisibility condition for reflexivity of $\Pq$ is satisfied.
\end{proof}

As a simple example, we consider when $q$ is supported by only one integer.

\begin{theorem}\label{thm:oneinteger}
If $\Pq$ is reflexive where $q$ is supported by $r=(r_1)$, then $r_1=1$.
For each $q$-vector of the form $q=(1,1,\ldots,1)$, we have $\Pq$ is IDP, reflexive, and $h^\ast$-unimodal.
\end{theorem}

\begin{proof}
If $q=(r_1,\ldots,r_1)$ with $r_1>1$, then $q$ fails the necessary divisibility condition to be reflexive.
Hence, we must have $r_1=1$.
In this case, since $q_j=1$ for all $j$, reflexivity is immediate by the divisibility condition.
It is straightforward to verify that the $h^\ast$-vector in this case is $(1,1,\ldots,1)$.
Finally, since every facet of $\Pq$ is a unimodular simplex, it follows that $\Pq$ is IDP.
\end{proof}

The following lemma is equivalent to Corollary~\ref{cor: reduction of theorem 4.1}.

\begin{lemma}
\label{lem:linearsystem}
If the $r$-vector $r = (r_1,\ldots,r_k)$ supports a reflexive IDP simplex $\Pq$ then there exists a vector of positive integers $x = (x_1,\ldots,x_k)$ satisfying the $k\times k$ system of linear equations $Rx=b$ where
\[
b_i:=r_i-1,
\]
and 
\[
R_{j,i}:=\left\{
\begin{array}{ll}
0 & \text{ if } i=j, \\
r_i & \text{ if } i<j, \\
r_i\bmod r_j & \text{ if } i>j.
\end{array} 
\right. \, 
\]
If a solution to this system of equations consisting of positive integers exists, then it corresponds to a reflexive $\Pq$.
However, this condition is necessary but not sufficient to establish IDP.
\end{lemma}

\begin{proof}
This lemma is a restatement of  Corollary~\ref{cor: reduction of theorem 4.1} using the notation $(r_1^{x_1},r_2^{x_2},\ldots,r_k^{x_k})$.
\end{proof}

It is important to observe that some reflexive IDP $\Pq$'s admit decompositions into reflexive IDP $\Pq$'s of smaller dimension, in the following sense.
Let $P,Q\subset\R^n$ be two lattice polytopes.  
We say that $P\oplus Q:=\conv{P\cup Q}$ is an \emph{affine free sum} if, up to unimodular equivalence, $P\cap Q = \{0\}$ and the affine span of $P$ and $Q$ are orthogonal coordinate subspaces of $\R^n$.
Suppose further that $P\subset \R^n$ and $Q\subset\R^m$ are reflexive polytopes with $0\in P$ and the vertices of $Q$ labeled as $v_0,v_1,\ldots,v_\ell$.  
For every $i=0,1,\ldots,m$, we define the polytope
\[
P\ast_i Q:=\conv{(P \times 0^m)\cup(0^n\times Q-v_i)}\subset\R^{n+m}.
\]
The following theorem indicates that affine free sum decompositions can be detected from the $q$-vector defining $\Pq$.
\begin{theorem}[Braun, Davis \cite{BraunDavisReflexive}]
\label{thm:freesumdecomp}
If $\Pp$ and $\Pq$ are full-dimensional reflexive simplices with $p=(p_1,\ldots,p_n)$ and $q=(q_1,\ldots,q_m)$, respectively, then $\Pp *_0 \Pq$ is a reflexive simplex $\Delta_{(1,y)}$ with $y=(p_1,\ldots,p_n,sq_1,\ldots,sq_m)$ where $s = 1+ \sum_{j=1}^n p_j$.
Moreover, if $\Delta_{(1,y)}$ arises in this form, then it decomposes as a free sum.
Further, if $\Pp$ and $\Pq$ are reflexive, IDP and $h^\ast$-unimodal, then so is $\Pp *_0 \Pq$.
\end{theorem}

\begin{remark}
If $P$ and $Q$ are lattice polytopes, then $P \times \{0\}$ will always be a facet of the affine free sum $P *_i Q$ for any $i$.
Now, if a polytope is IDP, then all of its facets are IDP.
Thus, if $P *_i Q$ is IDP, then $P$ must also be IDP.
However, it is possible to have two non-IDP, non-$h^\ast$-unimodal reflexive simplices produce a unimodal free sum through the $*_i$ operation: $\Pq$ with $q = (1,1,1,1,1,3)$ is not IDP and has a non-unimodal $h^\ast$-polynomial, but $\Pq *_0 \Pq$ has $q$-vector $(1,1,1,1,1,3,9,9,9,9,9,27)$ and $h^\ast$-vector $(1, 2, 5, 6, 10, 10, 13, 10, 10, 6, 5, 2, 1)$.
Moreover, the ordinary free sum of $\Pq$ with itself yields the same $h^\ast$-vector.
\end{remark}

Theorem~\ref{thm:freesumdecomp} shows that many reflexive, IDP, $h^\ast$-unimodal $\Pq$'s have these properties because they arise as affine free sums.
If a reflexive IDP $\Pq$ is an affine free sum of two reflexive IDP $\Pq$'s of smaller dimension, then in order to infer $h^\ast$-unimodality of the sum, this property of the summands must be separately verified.
Nonetheless, it is reasonable to focus attention on those reflexive IDP $\Pq$'s that do not arise as affine free sums, as these are fundamental examples that must be dealt with in any proof of Conjecture~\ref{conj:hibiohsugi}.
Fortunately, as the following two results indicate, for each $r$-vector there are at most finitely many such reflexive IDP simplices to consider.

\begin{prop}
\label{prop:freesumreflexive}
Fix an $r$-vector $r = (r_1,\ldots,r_k)$ for which $r_i\mid r_k$ for all $i\in\{1,\ldots,k-1\}$.
If $q=(r_1^{x_1},\ldots,r_k^{x_k})$ corresponds to a reflexive IDP simplex, then $\Pq$ is an affine free sum of the simplices defined by the $q$-vectors $(r_1^{x_1},\ldots,r_{k-1}^{x_{k-1}})$ and $(1^{x_k})$.
Further, the simplices defined by the $q$-vectors $(r_1^{x_1},\ldots,r_{k-1}^{x_{k-1}})$ and $(1^{x_k})$ are both IDP reflexive simplices.  
\end{prop}

\begin{proof}
Using the notation of Lemma~\ref{lem:linearsystem}, observe that after canceling denominators in $Rx=b$ our matrix equation has the form
\[
\left[
\begin{array}{ccccc|c}
0      & r_{2}\bmod r_{1} & r_{3}\bmod r_{1} & \cdots           & r_{k-1}\bmod r_{1} & 0 \\
r_1    & 0               & r_{3}\bmod r_{2} & \cdots           &  r_{k-1}\bmod r_{2} & \vdots \\
r_1    & r_2             & 0               & \ddots  &  \vdots           & \vdots \\
r_1    & r_2             & \ddots          & \ddots           & \vdots             & \vdots \\
\vdots & \vdots         &  \ddots            & \ddots           & r_{k-1}\bmod r_{k-2} & \vdots \\
r_1    & r_2             & r_3         & \cdots                 & 0             & 0 \\
\hline
r_1 & r_2 & r_3 & \cdots & r_{k-1} & 0
\end{array}
\right]
\left[
\begin{array}{c}
x_1 \\
x_2 \\
x_3 \\
\vdots \\
x_k
\end{array}
\right]
=
\left[
\begin{array}{c}
r_1-1 \\
r_2-1 \\
r_3-1 \\
\vdots \\
r_k-1
\end{array}
\right] \, .
\]
Notice that the upper-left block of this matrix and the first $k-1$ entries of the right-hand vector form the system from Lemma~\ref{lem:linearsystem} for the values $r_1<r_2<\cdots <r_{k-1}$.
Thus, the existence of a positive integer solution $(x_1,\ldots,x_k)$ to our matrix equation above implies that $(x_1,\ldots,x_{k-1})$ are the multiplicities of a $q$-vector supported by $(r_1,r_2,\ldots,r_{k-1})$.
Further, the final row of the matrix equation above implies that $r_k=1+\sum_{i=1}^{k-1}x_ir_i$, from which it follows that $q=(r_1^{x_1},\ldots,r_k^{x_k})$ arises from the free sum of $p=(r_1^{x_1},\ldots,r_{k-1}^{x_{k-1}})$ and $q=(1^{x_k})$ by Theorem~\ref{thm:freesumdecomp}.
Since each $r_i|r_k$, this also implies reflexivity for $p=(r_1^{x_1},\ldots,r_{k-1}^{x_{k-1}})$.

Theorem~\ref{thm:oneinteger} implies that the simplex defined by $(1^{x_k})$ is always reflexive and IDP.
It only remains to check that the simplex given by $(r_1^{x_1},\ldots,r_{k-1}^{x_{k-1}})$ is IDP.  
We know that $\Delta_{(1,q)} = \Delta_{(1,p)}\ast_0\Delta_{(1,h)}$ where we set $p:=(r_1^{x_1},\ldots,r_{k-1}^{x_{k-1}})$ and $h:=(1^{x_k})$.  
By the definition of affine free sums, we know the free sum $\Delta_{(1,p)}\ast_0\Delta_{(1,h)}$ is structured so that $\Delta_{(1,p)}$ is a face of $\Delta_{(1,q)}$.  
We know that $\Delta_{(1,q)} = \Delta_{(1,p)}\ast_0\Delta_{(1,h)}$ is IDP, and so each of its faces must also be IDP.  
\end{proof}

The final result in this section shows that for an $r$-vector such that $r_i \nmid r_k$ for some $i$, there are at most finitely many such IDP reflexive simplices supported on that vector. 
Each of these simplices may or may not be decomposable as an affine free sum.  

\begin{theorem}
\label{thm: finite reflexive}
For a fixed support vector $r$, there exist only finitely many reflexive and IDP $\Delta_{(1,q)}$ supported by $r$ that do not admit a free sum decomposition into lower dimensional $\Delta_{(1,q)}$.
\end{theorem}

\begin{proof}
Fix a support vector $r = (r_1,\ldots,r_k)$.  
By Proposition~\ref{prop:freesumreflexive} we know that if $r_i\,\mid\, r_k$ for all $i\in\{1,\ldots,k-1\}$, then any $\Delta_{(1,q)}$ supported by $r$ is an affine free sum of lower-dimensional $\Delta_{(1,q)}$.  
Thus, it remains to show that if $r_i\nmid r_k$ for at least one $i$ in $\{1,\ldots,k-1\}$ then there are at most finitely many reflexive and IDP $\Pq$ such that $q$ is supported by $r = (r_1,\ldots,r_k)$.  
To see this, notice that if $\Pq$ is reflexive and IDP and supported by $r$, then using the notation from Lemma~\ref{lem:linearsystem} there is a positive integer vector $x$ that satisfies $Rx=b$.
The set of $\Pq$ that are reflexive and IDP and supported by $r$ corresponds to the integer points in $\R_{>0}^k\cap \left( x+\ker(R)\right)$.
Since $\ker(R)$ is orthogonal to the image of $R^T$, it follows from $r_i\nmid r_k$ for some $i$ that there exists a strictly positive vector $a\in \textrm{im}(R^T)$ orthogonal to $\ker(R)$, namely the vector $a$ obtained as the sum of the transpose of the $k$-th row of $R$ and the transpose of the $i$-th row of $R$.
Hence, $\R_{>0}^k\cap \left( x+\ker(R)\right)$ is bounded, and therefore contains at most a finite number of integer points.
\end{proof}


\section{$q$-Vectors With Support on Two Integers}
\label{sec: support on two integers}

Using the results established in Sections~\ref{sec: idp for reflexive simplices} and~\ref{sec: q-vectors with fixed support}, we are now able to prove that all IDP and reflexive $\Delta_{(1,q)}$ supported on two integers are $h^\ast$-unimodal.
Let $q = (r_1^{x_1},r_2^{x_2})$ denote a $q$-vector supported on two vectors.  
By Lemma~\ref{lem:linearsystem}, in order to capture all IDP and reflexive simplices supported on two integers we will have two cases given by the possible divisibility relations between $r_1=r$ and $r_2=s$, as indicated in the following theorem.
\begin{theorem}
\label{thm: examples of sequences}
Let $r<s$ be positive integers and let
\[
q=(\underbrace{r,r,\ldots,r}_{m\text{ times}},\underbrace{s,s,\ldots,s}_{x\text{ times}})=(r^m,s^x) \, .
\]
Then $\Pq$ is reflexive if and only if $r\mid(1+sx)$ and $s\mid(1+rm)$.
Further, $\Pq$ is IDP and reflexive if and only if either 
\begin{enumerate}
\item $r\neq 1$ with $s=1+rm$ and $x=r-1$, or 
\item $r=1$ with $s=1+m$ and $x$ arbitrary.
\end{enumerate}
\end{theorem}

\begin{proof}
The claim regarding reflexivity follows immediately from the condition that $\Pq$ is reflexive if and only if $q_j\mid(1+q_1+\cdots+q_n)$ for all $j$.
To show that $\Delta_{(1,q)}$ being reflexive IDP implies $s=1+rm$ and $x=r-1$ when $r\neq 1$, we apply Corollary~\ref{cor: reduction of theorem 4.1}.
First, applying the corollary with $q_j=s$ we see that
\[
\frac{1}{s}+m\left\{\frac{r}{s} \right\}+(x-1)\left\{\frac{s}{s}\right\} = \frac{1+mr}{s}=1 \, ,
\]
and thus we have $1+mr=s$.
Second, setting $q_j=r$ and applying the corollary yields
\[
\frac{1}{r}+(m-1)\left\{\frac{r}{r} \right\}+x\left\{\frac{s}{r}\right\} = \frac{1}{r}+x\left\{\frac{1+mr}{r}\right\} = \frac{1+x}{r}=1 \, ,
\]
from which it follows that $x=r-1$.
A similar argument shows that if $r=1$, then $1+m=s$ and $x$ can be an arbitrary positive integer.

To show the converse, first observe that Theorem~\ref{thm:oneinteger} and Proposition~\ref{prop:freesumreflexive} tells us that if $q=(1^m,(1+m)^x)$ for any positive integer $x$, then $\Pq$ is reflexive IDP.
Thus, we are left with the case that $r\neq 1$ and $s=1+mr$ with $x=r-1$.
Corollary~\ref{cor: reduction of theorem 4.1} guarantees reflexivity; to verify that $\Pq$ is IDP in this case, we will directly apply Theorem~\ref{thm:icreflexives}.

First, consider the case where $q_j=r$, hence $b=1,\ldots,r-1$.
Since $1\leq b<r$, it follows that the left-hand side of~\eqref{eqn 1} reduces to
\[
b(rm-m+1)-(r-1)\left\lfloor\frac{brm+b}{r}\right\rfloor = b(rm-m+1)-(r-1)bm=b\, .
\]
Thus, our only option for the $c$ value in equations \eqref{eqn 2} and \eqref{eqn 3} is $c=1$.
We need to check that \eqref{eqn 2} holds when $q_i=rm+1$, $c=1$, and $b=2,3,\ldots,r-1$, which follows since
\[
m+(b-1)m=bm
\]
if and only if
\[
m+\left\lfloor \frac{(b-1)mr+(b-1)}{r}\right\rfloor = \left\lfloor\frac{brm+b}{r}\right\rfloor
\]
if and only if
\[
\left\lfloor\frac{rm+1}{r}\right\rfloor + \left\lfloor \frac{(b-1)(mr+1)}{r}\right\rfloor = \left\lfloor\frac{b(rm+1)}{r}\right\rfloor \, .
\]

Second, consider the case where $q_j=rm+1$, hence $b=1,2,\ldots,rm$.
For this value of $q_j$, the left-hand side of~\eqref{eqn 1} reduces to the following expression, which we denote by $h(b)$:
\begin{equation}\label{eqn 4}
h(b):=b-m\left\lfloor\frac{br}{mr+1}\right\rfloor
\end{equation}
Writing $b=km+\ell$ where $k,\ell \in \Z$, $0<\ell\leq m$, and $0\leq k<r$, we have
\begin{equation}\label{eqn:floorref}
\left\lfloor\frac{br}{mr+1}\right\rfloor=\left\lfloor\frac{(km+\ell)r}{mr+1}\right\rfloor=k+\left\lfloor\frac{r\ell-k}{mr+1}\right\rfloor=k \, .
\end{equation}
Thus, $h(b)=h(km+\ell)=\ell$.

Thus, the viable candidates for $c$-values in \eqref{eqn 2} and \eqref{eqn 3} are 
\[
c=1,m+1,2m+1,\ldots,(r-1)m+1 \, .
\]
To complete our proof, we observe that~\eqref{eqn 2} holds for all $b=mk+\ell$ with $2\leq \ell \leq m$ whenever $c=k'm+1$ for $k'\leq k$, since~\eqref{eqn:floorref} implies
\[
\left\lfloor\frac{(km+\ell)r}{mr+1}\right\rfloor-\left\lfloor\frac{(k'm+1)r}{mr+1}\right\rfloor = \ell-1 = \left\lfloor\frac{((k-k')m+\ell-1)r}{mr+1}\right\rfloor \, .
\]
\end{proof}

It remains to be shown that all IDP reflexive $\Delta_{(1,q)}$ supported on two integers are $h^\ast$-unimodal.  
\begin{theorem}
\label{thm: main theorem}
Let $\Delta_{(1,q)}$ be a reflexive simplex supported on two integers.  If $\Delta_{(1,q)}$ is IDP then it is $h^\ast$-unimodal.  
\end{theorem}

\begin{proof}
By Theorem~\ref{thm: examples of sequences}, we know that $\Delta_{(1,q)}$ with $q = (r^m,s^x)$ is reflexive and IDP if and only if either $r\neq 1$ with $s=1+rm$ and $x=r-1$ or $r=1$ with $s=1+m$ and $x$ arbitrary.  
In the latter case, $\Pq$ is the affine free sum $\Delta_{(1,q_1)}\ast_0\Delta_{(1,q_2)}$, where $q_1 = (1^m)$ and $q_2 = (1^x)$.
By Theorem~\ref{thm:oneinteger} and Theorem~\ref{thm:freesumdecomp}, it follows that $\Pq$ is $h^\ast$-unimodal.  
Thus, it only remains to prove $h^\ast$-unimodality for the former case.  

For the sake of clarity, in the following we will let $\Delta_{r,m}$ denote the simplex $\Delta_{(1,q)}$, where $q=(r^m,(1+rm)^{r-1})$ with $r\neq 1$.  
Recall that the formula given in Theorem~\ref{thm:hstarq} reduces to 
\[
h^\ast(\Delta_{r,m};z) =
\sum_{b=0}^{r(rm+1)-1}z^{\omega(b)},
\]
where
\[
\omega(b) = b-m\left\lfloor\frac{b}{rm+1}\right\rfloor-(r-1)\left\lfloor\frac{b}{r}\right\rfloor.
\]
To compute $h^\ast(\Delta_{r,m};z)$ as to reveal its unimodality, we first consider the sequence
\[
W:=\left(b-(r-1)\left\lfloor\frac{b}{r}\right\rfloor\right)_{b=0}^{r(rm+1)-1}.
\]
We will now group the terms in the sequence $W$ into subsequences of length $r$ and collect these into blocks of $m$ subsequences, since this will help make the desired unimodality of $h^\ast(\Delta_{r,m};z)$ more apparent.  
The grouping of terms into subsequences of length $r$ corresponds to mapping each term of $W$
\[
b- (r-1)\left\lfloor\frac{b}{r}\right\rfloor\longmapsto (i,j),
\]
where $i$ and $j$ are the unique nonnegative integers satisfying $b = ir+j$ and $0\leq j< r$.  
This partitions $W$ into subsequences given by
\begin{equation*}
\begin{split}
W 
&=  \left(b-(r-1)\left\lfloor\frac{b}{r}\right\rfloor\right)_{b=0}^{r(rm+1)-1}, \\
&=  \left(\left(i+j\right)_{j=0}^{r-1}\right)_{i=0}^{rm}. \\
\end{split}
\end{equation*}
The grouping of the subsequences of length $r$ into sequences of $m$ subsequences corresponds to the extending this map as
\[
b- (r-1)\left\lfloor\frac{b}{r}\right\rfloor\longmapsto (i,j,p),
\]
where $p=\left\lfloor\frac{i}{m}\right\rfloor$.
That is, for each $i$ we will write $i = pm + q$ for integers $0 \leq p < r$ and $0 \leq q < m$.
This lets us further decompose $W$ as
\begin{equation}
\label{eqn: partitioned W}
	W =  \left(\left(\left(\left(pm+q+j\right)_{j=0}^{r-1}\right)_{q=0}^{m-1}\right)_{p=0}^{r-1},\left(rm+j\right)_{j=0}^{r-1}\right).
\end{equation}
To compute the powers $z^{\omega(b)}$ from these terms, we must subtract the value $m\left\lfloor\frac{b}{rm+1}\right\rfloor$ from each term in $W$. 
Doing so, we will write the corresponding sequence as
\[
	W' := \left(b-(r-1)\left\lfloor\frac{b}{r}\right\rfloor - m\left\lfloor\frac{b}{rm+1}\right\rfloor\right)_{b=0}^{r(rm+1)-1}.
\]
Notice that, since $b = (pm+q)r + j$, we also have
\[
\begin{split}
	\left\lfloor\frac{b}{rm+1}\right\rfloor &= \left\lfloor\frac{p(rm+1)+qr+j-p}{rm+1}\right\rfloor, \\
		&= p + \left\lfloor\frac{qr+j-p}{rm+1}\right\rfloor, \\
		&= \begin{cases}
			p-1 & \text{ if } q=0 \text{ and }  j<p ,\\
			p & \text{ otherwise}.
			\end{cases}
\end{split}
\]
Thus, applying this fact together with the expression for $W$ given in equation~\eqref{eqn: partitioned W}, we have that
\[
	W' = \left(\left(\left((q+j)_{j=0}^{r-1}\right)_{q=1}^{m-1}\right)_{p=1}^{r-1}, \left(\left((m+j)_{j=0}^{p-1},(j)_{j=p}^{r-1}\right)\right)_{p=1}^{r-1}, \left((q+j)_{j=0}^{r-1}\right)_{q=0}^{m-1},(m+j)_{j=0}^{r-1}\right).
\]

It follows that the powers $z^{\omega(b)}$ sum to yield the $h^\ast$-polynomial as
\[
\begin{split}
	h^\ast(\Delta_{r,m};z) &= \sum_{w \in W'} z^w, \\
		&= \sum_{p=1}^{r-1} \left(\sum_{q=1}^{m-1} z^q\right)\left(\sum_{j=0}^{r-1} z^j\right) + \sum_{p=1}^{r-1}\left(\sum_{j=0}^{p-1} z^{m+j} + \sum_{j=p}^{r-1} z^j\right),\\
		&\quad + \left(\sum_{q=0}^{m-1} z^q\right)\left(\sum_{j=0}^{r-1} z^j\right) + \sum_{j=0}^{r-1} z^{m+j}, \\
		&= (r-1)\left(\sum_{q=1}^{m-1} z^q\right)\left(\sum_{j=0}^{r-1} z^j\right) + \sum_{k=0}^{r-2} (r-1-k)z^{m-k} + \sum_{k=1}^{r-1} kz^k,  \\
		&\quad + \left(\sum_{q=1}^{m-1} z^q\right)\left(\sum_{j=0}^{r-1} z^j\right) + \sum_{j=0}^{r-1} z^j + \sum_{j=0}^{r-1} z^{m+j}, \\
		&= r\left(\sum_{q=1}^{m-1} z^q\right)\left(\sum_{j=0}^{r-1} z^j\right) + \sum_{j=0}^{r-1} (r-j)z^{m+j} + \sum_{j=0}^{r-1} (j+1)z^j.
\end{split}
\]

Now, let
\[
	p(z) = \sum_{i=0}^{r+m-1} p_iz^i := r\left(\sum_{q=1}^{m-1} z^q\right)\left(\sum_{j=0}^{r-1} z^j\right),
\]
and
\[
	q(z) = \sum_{i=0}^{r+m-1} q_iz^i :=  \sum_{j=0}^{r-1} (r-j)z^{m+j} + \sum_{j=0}^{r-1} (j+1)z^j.
\]
Since $h^\ast(\Delta_{(r,n)};z)$ is symmetric with respect to degree $r+m-1$, then to see that it is unimodal it suffices to observe that the $(r-1)^{st}$ coefficient is at most the $r^{th}$ coefficient.  
To see this, first note that $p(z)$ is also symmetric with respect to degree $r+m-1$ and unimodal. 
We then consider the following two cases: first, suppose $q(z)$ has no internal zeros; that is, there are no three indices $0 \leq i < j < k \leq r+m-1$ such that $q_i,q_k \neq 0$ and $q_j = 0$.
Then $q(z)$ is also symmetric with respect to degree $r+m-1$ and unimodal. 
Thus, $h^{\ast}(\Delta_{r,m};z)$ is unimodal. 
Second, suppose $q(z)$ does have internal zeros. 
This forces $q_{r-1} = r$ while $q_r = 0$. 
However, we know that $r+p_{r-1} \leq p_r$.  
Since $q(z)$ is symmetric with respect to degree $r+m-1$, we have that $h^{\ast}(\Delta_{r,m};z)$ is again unimodal. 
This completes the proof.
\end{proof}

\section{Future Directions} \label{subsec: future directions}

The results of this paper set the stage for a detailed investigation of Conjecture~\ref{conj:hibiohsugi} in the special case of the reflexive simplices of the form $\Delta_{(1,q)}$.  
As such, these results provide some natural directions for future work in discrete geometry as outlined by the following questions.

First, the vector $q=(3,20,24,24,24,24)$ satisfies Lemma~\ref{lem:linearsystem}, yet is non-IDP and has a non-unimodal $h^\ast$-vector, specifically 
\[
h^\ast(\Pq;z)=z^6+16z^5+29z^4+28z^3+29z^2+16z+1\, .
\]
However, it is not clear in general what role Lemma~\ref{lem:linearsystem} plays regarding $h^\ast$-unimodality and the structure of Hilbert bases for $\Pq$, leading to the following question.
\begin{question} 
\label{quest: linear system}
	Suppose $r = (r_1,\ldots,r_k)$ supports a reflexive simplex $\Pq$ with $q$-vector satisfying Lemma~\ref{lem:linearsystem}.
        What constraints on the Hilbert basis of $\cone{\Pq}$, if any, are implied by the linear system in Lemma~\ref{lem:linearsystem}?
        Further, for which divisibility patterns in the $r$-vector, if any, is Lemma~\ref{lem:linearsystem} sufficient to imply $h^\ast$-unimodality for $\Pq$?
        Which families of $q$-vectors, if any, both satisfy Lemma~\ref{lem:linearsystem} and fail to be $h^\ast$-unimodal?
\end{question}

Second, in Section~\ref{sec: support on two integers}, $h^\ast$-unimodality for some, but not all, of the cases studied follows from an affine free sum decomposition of the simplices.
Given this, it is natural to ask when the free sum construction underlies $h^\ast$-unimodality in the context of Section~\ref{sec: q-vectors with fixed support}.
\begin{question}
\label{quest: free sums}
	For each fixed dimension $d$, what fraction of reflexive IDP simplices of the form $\Pq$ arise via affine free sums?
\end{question}
Third, the reflexive simplices satisfying the necessary condition of Lemma~\ref{lem:linearsystem} correspond to positive integral solutions to multivariate systems of algebraic equations.  
A full understanding of these simplices may require tools from algebraic geometry and algebraic number theory.
\begin{question}
\label{quest: number theory}
What are the integral solutions to the multivariate system of algebraic equations arising from Lemma~\ref{lem:linearsystem}?
\end{question}

Fourth, as far as the authors know, in all published proofs that a family of polytopes is both IDP and $h^\ast$-unimodal, IDP and $h^\ast$-unimodality are established separately.
Thus, the IDP has only been observed to correlate with $h^\ast$-unimodality; no proofs have shown a causal link between the two properties.
For example, in well-known work of Bruns and R\"omer~\cite{BrunsRomer}, the IDP condition is a consequence of the existence of a regular unimodular triangulation, while $h^\ast$-unimodality is established by using that triangulation to induce an application of the $g$-theorem.
Thus, there is a need for proofs that demonstrate explicitly how the IDP condition might be used to prove $h^\ast$-unimodality directly.
The following question is a step in this direction.
\begin{question}
\label{quest:IDPimpliesh*}
Do there exist $q$-vectors for which $h^\ast$-unimodality of $\Pq$ can be established using only the conditions of Theorem~\ref{thm:icreflexives}?
\end{question}

As a final remark, the results of this paper offer new avenues by which to search for counterexamples to Conjecture~\ref{conj:hibiohsugi}.  
In the special case of the simplices $\Pq$, the expression of the $h^\ast$-polynomial arising from Theorem~\ref{thm:hstarq} provides a direct link between the structure of a $q$-vector and the behavior of the associated $h^\ast$-polynomial.
A deeper study of general polynomials of the form given in Theorem~\ref{thm:hstarq} might provide insight leading to a counterexample, if one exists.
Another possible direction in which to search for counterexamples is suggested by Question~\ref{quest: linear system}, through the identification of broad families of $q$-vectors that simultaneously satisfy Lemma~\ref{lem:linearsystem} and fail to be $h^\ast$-unimodal.
Even if a counterexample is found to the general conjecture, improved explanations of the broadly-observed $h^\ast$-unimodality for reflexive polytopes, including in the special case of $\Pq$, is needed.

\bibliographystyle{plain}
\bibliography{Braun}

\begin{thebibliography}{10}

\bibitem{beckjayawantmcallister}
Matthias Beck, Pallavi Jayawant, and Tyrrell~B. McAllister.
\newblock Lattice-point generating functions for free sums of convex sets.
\newblock {\em J. Combin. Theory Ser. A}, 120(6):1246--1262, 2013.

\bibitem{BeckRobinsCCD}
Matthias Beck and Sinai Robins.
\newblock {\em Computing the continuous discretely}.
\newblock Undergraduate Texts in Mathematics. Springer, New York, second
  edition, 2015.
\newblock Integer-point enumeration in polyhedra, With illustrations by David
  Austin.

\bibitem{BraunUnimodalSurvey}
Benjamin Braun.
\newblock Unimodality problems in {E}hrhart theory.
\newblock In {\em Recent trends in combinatorics}, volume 159 of {\em IMA Vol.
  Math. Appl.}, pages 687--711. Springer, [Cham], 2016.

\bibitem{BraunDavisReflexive}
Benjamin Braun and Robert Davis.
\newblock Ehrhart series, unimodality, and integrally closed reflexive
  polytopes.
\newblock {\em Ann. Comb.}, 20(4):705--717, 2016.

\bibitem{brentisurvey}
Francesco Brenti.
\newblock Log-concave and unimodal sequences in algebra, combinatorics, and
  geometry: an update.
\newblock In {\em Jerusalem combinatorics '93}, volume 178 of {\em Contemp.
  Math.}, pages 71--89. Amer. Math. Soc., Providence, RI, 1994.

\bibitem{BrunsRomer}
Winfried Bruns and Tim R{\"o}mer.
\newblock {$h$}-vectors of {G}orenstein polytopes.
\newblock {\em J. Combin. Theory Ser. A}, 114(1):65--76, 2007.

\bibitem{conrads}
Heinke Conrads.
\newblock Weighted projective spaces and reflexive simplices.
\newblock {\em Manuscripta Math.}, 107(2):215--227, 2002.

\bibitem{Ehrhart}
Eug{\`e}ne Ehrhart.
\newblock Sur les poly\`edres rationnels homoth\'etiques \`a {$n$}\ dimensions.
\newblock {\em C. R. Acad. Sci. Paris}, 254:616--618, 1962.

\bibitem{M2}
Daniel~R. Grayson and Michael~E. Stillman.
\newblock Macaulay2, a software system for research in algebraic geometry.
\newblock Available at \url{http://www.math.uiuc.edu/Macaulay2/}.

\bibitem{hibiflawless}
Takayuki Hibi.
\newblock Flawless {$O$}-sequences and {H}ilbert functions of
  {C}ohen-{M}acaulay integral domains.
\newblock {\em J. Pure Appl. Algebra}, 60(3):245--251, 1989.

\bibitem{Hibi}
Takayuki Hibi.
\newblock {\em Algebraic Combinatorics on Convex Polytopes}.
\newblock Carslaw Publications, Australia, 1992.

\bibitem{HibiDualPolytopes}
Takayuki Hibi.
\newblock Dual polytopes of rational convex polytopes.
\newblock {\em Combinatorica}, 12(2):237--240, 1992.

\bibitem{MustataPayne}
Mircea Musta{\c{t}}{\v{a}} and Sam Payne.
\newblock Ehrhart polynomials and stringy {B}etti numbers.
\newblock {\em Math. Ann.}, 333(4):787--795, 2005.

\bibitem{NillSimplices}
Benjamin Nill.
\newblock Volume and lattice points of reflexive simplices.
\newblock {\em Discrete Comput. Geom.}, 37(2):301--320, 2007.

\bibitem{hibiohsugiconj}
Hidefumi Ohsugi and Takayuki Hibi.
\newblock Special simplices and {G}orenstein toric rings.
\newblock {\em J. Combin. Theory Ser. A}, 113(4):718--725, 2006.

\bibitem{Payne}
Sam Payne.
\newblock Ehrhart series and lattice triangulations.
\newblock {\em Discrete Comput. Geom.}, 40(3):365--376, 2008.

\bibitem{StanleyDecompositions}
Richard~P. Stanley.
\newblock Decompositions of rational convex polytopes.
\newblock {\em Ann. Discrete Math.}, 6:333--342, 1980.
\newblock Combinatorial mathematics, optimal designs and their applications
  (Proc. Sympos. Combin. Math. and Optimal Design, Colorado State Univ., Fort
  Collins, Colo., 1978).

\bibitem{stanleylogconcave}
Richard~P. Stanley.
\newblock Log-concave and unimodal sequences in algebra, combinatorics, and
  geometry.
\newblock In {\em Graph theory and its applications: {E}ast and {W}est
  ({J}inan, 1986)}, volume 576 of {\em Ann. New York Acad. Sci.}, pages
  500--535. New York Acad. Sci., New York, 1989.

\end{thebibliography}

\end{document}